\documentclass[11pt, a4paper]{article}

\usepackage{color,url,latexsym,amssymb,amsthm,amsmath,amsxtra,amsfonts}
\usepackage{graphicx}
\usepackage{amsfonts}
\usepackage{amsbsy}
\usepackage{amssymb}
\usepackage{amsmath}
\usepackage{amsthm}
\theoremstyle{plain}

\newtheorem{corollary}{Corollary}[section]
\newtheorem{proposition}{Proposition}[section]
\theoremstyle{definition}
\newtheorem{definition}{Definition}[section]
\newtheorem{remark}{Remark}[section]

\begin{document}

\title{Adapted basic connections to a certain subfoliation on the tangent manifold of a Finsler space}
\author{Adelina Manea and Cristian Ida}
\date{}
\maketitle

\begin{abstract}
On the slit tangent manifold $TM^0$ of a Finsler space $(M,F)$ there are given some natural foliations as vertical foliation and some other fundamental foliations produced by the vertical and horizontal Liouville vector fields, see [ A. Bejancu, H. R. Farran, {\em Finsler Geometry and Natural Foliations on the Tangent Bundle}, Rep. Math. Physics \textbf{58}, No. 1 (2006), 131-146]. In this paper we consider a $(n,2n-1)$-codimensional subfoliation $(\mathcal{F}_V,\mathcal{F}_{\Gamma})$ on $TM^0$ given by vertical foliation $\mathcal{F}_V$ and the line foliation spanned by vertical Liouville vector field $\Gamma$ and we give a triplet of basic connections adapted to this subfoliation.
\end{abstract}

\medskip

\begin{flushleft}
\strut \textbf{2010 Mathematics Subject Classification}: 53B40, 53C12, 53C60.

\textbf{Key Words}: Finsler manifold, subfoliation, basic connection.
\end{flushleft}

\section{Introduction and preliminaries}
\setcounter{equation}{0}

\subsection{Introduction}

The Finsler manifolds are interesting models for some physical phenomena, so their properties are useful to be investigate \cite{B-C-S, B-M, M-A}.
On the other hand in the paper \cite{Bej} Bejancu and Farran have initiated a study of interrelations between the geometry of foliations on the tangent manifold of a Finsler space and the geometry of the Finsler space itself. The main idea of their paper is to emphasize the importance of some foliations which exist on the tangent bundle of a Finsler space $(M,F)$, in studying the differential geometry of $(M,F)$ itself. In this direction, in the last decades, the geometrical aspects determined by these foliations on the tangent manifold of a Finsler space were studied \cite{A-R}, \cite{C-C}, \cite{Ma}, \cite{P-T-Z}.  The foliated manifolds, and some couple of foliations, one of them being a subfoliation of the other, are also studied, related sometimes with cohomological theories, \cite{C-M}, \cite{D1}, \cite{D}, \cite{T}. Our present work intends to develop the study of the Finsler spaces and the foliated structures on the tangent manifold of such a space.

The paper is organized as follows: In the preliminary subsection we recall some basic facts on Finsler spaces and we present some natural foliations on the tangent manifold $TM^0$ of a Finsler space $(M,F)$, according to \cite{B-F}, \cite{Bej}. In the second section, using the vertical Liouville vector field $\Gamma$ and the natural almost complex structure $J$ on $TM^0$ we give an adapted basis in $T(TM^0)$. In the last section are given the main results of the paper. There is identified a $(n,2n-1)$-codimensional subfoliation $(\mathcal{F}_V,\mathcal{F}_{\Gamma})$ on the tangent manifold $TM^0$ of a Finsler space $(M,F)$, where $\mathcal{F}_V$ and $\mathcal{F}_{\Gamma}$ are the vertical foliation and the line foliation spanned by $\Gamma$, respectively. Firstly we make a general approach about basic connections on the normal bundles related to this subfoliation and next a triple of adapted basic connections with respect to this subfoliation is constructed. The methods used here are similarly  and closely related to those used in \cite{C-M} for the general study of $(q_1,q_2)$-subfoliations.

\subsection{Preliminaries and notations}
Let $(M,F)$ be a $n$-dimensional Finsler manifold with $(x^i, y^i),\,i=1,\ldots,n$ the local coordinates on $TM$ (for necessary definitions see for instance \cite{A-P, B-C-S, M-A}).

The vertical bundle $V(TM^0)$ of $TM^0=TM-\{{\rm zero\,\,section}\}$ is the tangent (structural) bundle to vertical foliation $\mathcal{F}_{V}$ determined by the fibers of $\pi:TM\rightarrow M$ and characterized by $x^k=const.$ on the leaves. Also, we locally have $V(TM^0)={\rm span}\left\{\frac{\partial}{\partial y^i}\right\},\,i=1,\ldots,n$.

A canonical transversal (also called horizontal) distribution to $V(TM^0)$ is constructed by Bejancu and Farran in \cite{B-F} pag. 225 or \cite{Bej} as
follows:

Let $(g^{ji}(x,y))_{i\times j}$ be the inverse matrix of $(g_{ij}(x,y))_{i\times j}$, where
\begin{equation}
g_{ij}(x,y)=\frac{1}{2}\frac{\partial^2F^2}{\partial y^i\partial y^j}(x,y),
\label{1}
\end{equation}
and $F$ is the fundamental function of the Finsler space.

If consider the local functions
\begin{equation}
G^i=\frac{1}{4}g^{ik}\left(\frac{\partial^2F^2}{\partial y^k\partial x^h}y^h-\frac{\partial F^2}{\partial x^k}\right),\,G^j_i=\frac{\partial G^j}{\partial y^i},
\label{2}
\end{equation}
then, there exists on $TM^0$ a $n$-distribution $H(TM^0)$ locally spanned by the vector fields
\begin{equation}
\frac{\delta}{\delta x^i}=\frac{\partial}{\partial x^i}-G^j_i\frac{\partial}{\partial y^j},\,i=1,\ldots,n.
\label{3}
\end{equation}
The local basis $\left\{\frac{\delta}{\delta x^i}, \frac{\partial}{\partial y^i}\right\},\,i=1,\ldots,n$ is called \textit{adapted} to vertical foliation $\mathcal{F}_{V}$ and we have the decomposition
\begin{equation}
T(TM^0)=H(TM^0)\oplus V(TM^0).
\label{4}
\end{equation}
If we consider the dual adapted bases $\{dx^i, \delta y^i=dy^i+G^i_jdx^j\}$, then the Riemannian metric $G$ on $TM^0$ given by the Sasaki lift of the fundamental metric tensor $g_{ij}$ from ({\ref{1}}) satisfies
\begin{equation}
G\left(\frac{\delta}{\delta x^i}, \frac{\delta}{\delta x^j}\right)=G\left(\frac{\partial}{\partial y^i}, \frac{\partial}{\partial y^j}\right)=g_{ij}\,,\,G\left(\frac{\delta}{\delta x^i}, \frac{\partial}{\partial y^j}\right)=0, \,i,j=1,\ldots,n.
\label{5}
\end{equation}
We also notice that there is  a natural almost complex structure on $TM^0$ which is compatible with $G$ and locally defined by
\begin{displaymath}
J=\frac{\delta}{\delta x^i}\otimes \delta y^i-\frac{\partial}{\partial y^i}\otimes dx^i\,,\,J\left(\frac{\delta}{\delta x^i}\right)=-\frac{\partial}{\partial y^i}\,,\,J\left(\frac{\partial}{\partial y^i}\right)=\frac{\delta}{\delta x^i}.
\end{displaymath}
According to \cite{B-M, M-A} we have that $(TM^0,G,J)$ is an almost K\"{a}hlerian manifold with the almost K\"{a}hler form given by
\begin{equation}
\Omega(X,Y)=G(JX,Y),\,\,\forall\,X,Y\in\mathcal{X}(TM^0).
\label{k1}
\end{equation}
It is easy to see that locally we have
\begin{equation}
\Omega=g_{ij}\delta y^i\wedge dx^i.
\label{k2}
\end{equation}

Now, let us consider $\Gamma=y^i\frac{\partial}{\partial y^i}$ the \textit{vertical Liouville vector field} (or radial vertical vector field) which is globally defined on $TM^0$. We also  consider $\mathcal{L}_{\Gamma}=span\,\{\Gamma\}$ the line distribution spanned by $\Gamma$ on $TM^0$ and the following complementary orthogonal distributions to $\mathcal{L}_{\Gamma}$ in $V(TM^0)$ and $T(TM^0)$, respectively
\begin{equation}
\mathcal{L}_{\Gamma}^{'}=\{X\in\Gamma(V(TM^0))\,:\,G(X,\Gamma)=0\},
\label{6}
\end{equation}
\begin{equation}
\mathcal{L}_{\Gamma}^{\perp}=\{X\in\Gamma(T(TM^0))\,:\,G(X,\Gamma)=0\}.
\label{7}
\end{equation}
In \cite{Bej} it is proved that both distributions $\mathcal{L}_{\Gamma}^{'}$ and $\mathcal{L}_{\Gamma}^{\perp}$ are integrable and we also have the decomposition
\begin{equation}
V(TM^0)=\mathcal{L}_{\Gamma}^{'}\oplus\mathcal{L}_{\Gamma}.
\label{8}
\end{equation}
Moreover, we have $\mathcal{L}_{\Gamma}^{\perp}=H(TM^0)\oplus\mathcal{L}_{\Gamma}^{'}$ and the following result:
\begin{proposition}
\begin{enumerate}
\item[i)] The foliation $\mathcal{F}_{\Gamma}^{\perp}$ determined by the distribution $\mathcal{L}_{\Gamma}^{\perp}$ is just the foliation determined by the level hypersurfaces of the fundamental function $F$ of the Finsler manifold, denoted by $\mathcal{F}_F$ and called \textit{the fundamental foliation} on $(TM^0,G)$.
\item[ii)] For every fixed point $x_0\in M$, the leaves of the vertical Liouville foliation $\mathcal{F}_{\Gamma}^{'}$ determined by the distribution $\mathcal{L}_{\Gamma}^{'}$ on $T_{x_0}M$ are just the $c$-indicatrices of $(M,F)$:
\begin{equation}
I_{x_0}M(c)=\{y\in T_{x_0}M\,:\,F(x_0,y)=c\}.
\label{9}
\end{equation}
\item[iii)] The foliation $\mathcal{F}_{\Gamma}^{'}$ is a subfoliation of the vertical foliation $\mathcal{F}_{V}$.
\end{enumerate}
\end{proposition}

\section{An adapted basis in $T(TM^0)$}
\setcounter{equation}{0}
In this section, using the vertical Liouville vector field $\Gamma$ and the natural almost complex structure $J$ on $TM^0$, we give an adapted basis in $T(TM^0)$.

There are some useful facts which follow from the homogeneity condition of the fundamental function of the Finsler manifold $(M,F)$. By the Euler theorem on positively homogeneous functions we have, \cite{Bej}:
\begin{equation}
	F^2=y^iy^jg_{ij},\, \frac{\partial F}{\partial y^k}=\frac{1}{F}y^ig_{ki}, \, y^i\frac{\partial g_{ij}}{\partial y^k}=0,\,  k=1,\ldots,n.
\label{10}
\end{equation}
Hence it results
\begin{equation}
G(\Gamma,\Gamma)=F^2.
\label{11}
\end{equation}
As we already saw, the vertical bundle is locally spanned by $\left\{\frac{\partial}{\partial y^i}\right\},\,i=1,\ldots,n$ and it admits decomposition \eqref{8}. In the sequel, according to \cite{Ma1},  we give another basis on $V(TM^0)$, adapted to $\mathcal{F}_{\Gamma}^{'}$, and following \cite{A-R}, \cite{Ma1} we extend this basis to an adapted basis in $T(TM^0)$.

We consider the following vertical vector fields:
\begin{equation}
\frac{\overline{\partial}}{\overline{\partial} y^k}=\frac{\partial}{\partial y^k}-t_k\Gamma,\, k=1,\ldots,n,
\label{II1}
\end{equation}
where functions $t_k$ are defined by the conditions
\begin{equation}
G\left(\frac{\overline{\partial}}{\overline{\partial} y^k},\Gamma\right)=0, \forall\, k=1,\ldots,n.
\label{II2}
\end{equation}
The above conditions become
\begin{displaymath}
G\left(\frac{\partial}{\partial y^k}, y^i\frac{\partial}{\partial y^i}\right)-t_kG(\Gamma,\Gamma)=0
\end{displaymath}
so, taking into account also \eqref{5} and \eqref{11}, we obtain the local expression of functions $t_k$ in a local chart $(U,(x^i,y^i))$:
\begin{equation}
t_k=\frac{1}{F^2}y^ig_{ki}=\frac{1}{F}\frac{\partial F}{\partial y^k},\,\forall\, k=1,\ldots,n.
\label{II3}
\end{equation}
If $(\widetilde{U},(\widetilde{x}^{i_1},\widetilde{y}^{i_1}))$ is another local chart on $TM^0$, in $U \cap \widetilde{U} \neq \emptyset $, then we have:
\begin{displaymath}	\widetilde{t}_{k_1}=\frac{1}{F^2}\widetilde{y}^{i_1}\widetilde{g}_{i_1k_1}=\frac{1}{F^2}\frac{\partial \widetilde{x}^{i_1}}{\partial x^i}y^i \frac{\partial x^k}{\partial \widetilde{x}^{k_1}}\frac{\partial x^j}{\partial \widetilde{x}^{i_1}}g_{kj}=\frac{\partial x^k}{\partial \widetilde{x}^{k_1}}t_k.
\end{displaymath}
So, we obtain the following changing rule for the vector fields \eqref{II1}:
\begin{equation}
\frac{\overline{\partial}}{\overline{\partial} \widetilde{y}^{i_1}}=\frac{\partial x^k}{\partial \widetilde{x}^{i_1}}\frac{\overline{\partial}}{\overline{\partial} y^k},\, \forall\, i_1=1,\ldots, n.
	\label{II4}
\end{equation}
By a straightforward calculation, using \eqref{10}, it results:
\begin{proposition}
\label{p2.1}
(\cite{Ma1}). The functions $\{t_k\},\,k=1,\ldots,n$ defined by \eqref{II3} are satisfying:
\begin{enumerate}
	\item[i)] $y^it_i=1,\, y^i \frac{\overline{\partial}}{\overline{\partial} y^i}=0$;
\item[ii)] $\frac{\partial t_l}{\partial y^k}=-2t_kt_l+\frac{1}{F^2} g_{kl},\, \Gamma t_k=-t_k,\, \forall\, k,l=1,\ldots,n$;
\item[iii)]	$ y^j\frac{\partial t_j}{\partial y^i}=-t_i,\, \forall\, i=1,\ldots,n,\, y^i(\Gamma t_i)=-1$.
\end{enumerate}
\end{proposition}
\begin{proposition}
\label{p2.2}
(\cite{Ma1}). The following relations hold:
\begin{equation}
\left[\frac{\overline{\partial}}{\overline{\partial} y^i},\frac{\overline{\partial}}{\overline{\partial} y^j}\right]=t_i\frac{\overline{\partial}}{\overline{\partial} y^j}-t_j\frac{\overline{\partial}}{\overline{\partial} y^i},\,\,
\left[\frac{\overline{\partial}}{\overline{\partial} y^i},\Gamma\right]=\frac{\overline{\partial}}{\overline{\partial} y^i},
\label{II5}
\end{equation}
for all $i,j=\overline{1,n}$.
\end{proposition}

By conditions \eqref{II2}, the vector fields $\left\{\frac{\overline{\partial}}{\overline{\partial} y^1},\ldots,\frac{\overline{\partial}}{\overline{\partial} y^n}\right\}$ are orthogonal to $\Gamma$, so they belong to the $(n-1)$-dimensional distribution $\mathcal{L}^{'}_{\Gamma}$. It results that they are linear dependent and, from Proposition \ref{p2.1} i), we have 
\begin{equation}
\frac{\overline{\partial}}{\overline{\partial} y^n}=-\frac{1}{y^n}y^a\frac{\overline{\partial}}{\overline{\partial} y^a},
\label{II6}
\end{equation}
since the local coordinate $y^n$ is nonzero everywhere. 

We also proved that, \cite{Ma1}:
\begin{proposition}
\label{p2.3}
The system $\left\{\frac{\overline{\partial}}{\overline{\partial} y^1},\ldots,\frac{\overline{\partial}}{\overline{\partial} y^{n-1}},\Gamma\right\}$ of vertical vector fields is a locally adapted basis to the vertical Liouville foliation $\mathcal{F}_{\Gamma}^{'}$, on $V(TM^0)$.
\end{proposition}
More clearly, let $\left(\widetilde{U},(\widetilde{x}^{i_1},\widetilde{y}^{i_1})\right), \left(U,(x^i,y^i)\right)$ be two local charts which domains overlap, where $\widetilde{y}^k$ and $y^n$ are nonzero functions (in every local charts on $TM^0$ there is at least one nonzero coordinate function $y^i$). 

The adapted basis in $\widetilde{U}$ is $\left\{\frac{\overline{\partial}}{\overline{\partial}\widetilde{y}^1},\ldots,\frac{\overline{\partial}}{\overline{\partial}\widetilde{y}^{k-1}},\frac{\overline{\partial}}{\overline{\partial}\widetilde{y}^{k+1}},\ldots,\frac{\overline{\partial}}{\overline{\partial}\widetilde{y}^n},\Gamma\right\}$. In $U\cap \widetilde{U}$ we have relations \eqref{II4} and \eqref{II6}, hence
\begin{displaymath}
\frac{\overline{\partial}}{\overline{\partial}\widetilde{y}^{i_1}}=\sum_{i=1}^{n-1}\left(\frac{\partial x^i}{\partial \widetilde{x}^{i_1}}-\frac{y^i}{y^n}\frac{\partial x^n}{\partial \widetilde{x}^{i_1}}\right)\frac{\overline{\partial}}{\overline{\partial} y^i},\, \frac{\overline{\partial}}{\overline{\partial} y^j}=\sum_{j_1=1, j_1 \neq k}^n\left(\frac{\partial \widetilde{x}^{j_1}}{\partial x^j}-\frac{\widetilde{y}^{j_1}}{\widetilde{y}^k}\frac{\partial \widetilde{x}^k}{\partial x^j}\right)\frac{\overline{\partial}}{\overline{\partial} \widetilde{y}^{j_1}},
\end{displaymath}
for all $i_1=1,\ldots,n$, $i_1\neq k$, $j=1,\ldots,n-1$. We can see that the above relations also imply
\begin{displaymath}
\frac{\partial x^i}{\partial \widetilde{x}^{k}}-\frac{y^i}{y^n}\frac{\partial x^n}{\partial \widetilde{x}^{k}}=-\sum_{i_1=1,i_1 \neq k}^{n}\frac{\widetilde{y}^{i_1}}{\widetilde{y}^k}\left(\frac{\partial x^i}{\partial \widetilde{x}^{i_1}}-\frac{y^i}{y^n}\frac{\partial x^n}{\partial \widetilde{x}^{i_1}}\right).
\end{displaymath}
By a straightforward calculation we have that the changing matrix of basis 
\begin{displaymath}
\left\{\frac{\overline{\partial}}{\overline{\partial} y^1},\ldots,\frac{\overline{\partial}}{\overline{\partial} y^{n-1}},\Gamma\right\} \rightarrow \left\{\frac{\overline{\partial}}{\overline{\partial} \widetilde{y}^1},\ldots,\frac{\overline{\partial}}{\overline{\partial} \widetilde{y}^{k-1}}, \frac{\overline{\partial}}{\overline{\partial} \widetilde{y}^{k+1}},\ldots,\frac{\overline{\partial}}{\overline{\partial} \widetilde{y}^n},\Gamma\right\}
\end{displaymath}
on $V(TM^0)=\mathcal{L}_{\Gamma}'\oplus \mathcal{L}_{\Gamma}$ has the determinant equal to 
$(-1)^{n+k}\frac{\widetilde{y}^k}{y^n} \det\left(\frac{\partial x^i}{\partial \widetilde{x}^j}\right)_{i,j=\overline{1,n}}$.

Also, we can denote 
\begin{displaymath}
\frac{\overline{\partial}}{\overline{\partial}y^a}=E_a^i\frac{\partial}{\partial y^i},\,a =1,\ldots,n-1,
\end{displaymath}
where $rank\,E_a^i=n-1$ and $E_a^iy^jg_{ij}=0$.

Now, using the natural almost complex structure $J$ on $T(TM^0)$, the new local vector field frame in $T(TM^0)$ is 
\begin{equation}
\left\{\frac{\overline{\delta}}{\overline{\delta} x^a}, \xi,\frac{\overline{\partial}}{\overline{\partial} y^a},\Gamma\right\},
\label{II7}
\end{equation} 
where 
\begin{displaymath}
\xi=J\Gamma=y^i\frac{\delta}{\delta x^i}\,,\,\frac{\overline{\delta}}{\overline{\delta} x^a}=J\frac{\overline{\partial}}{\overline{\partial} y^a}=E_a^i\frac{\delta}{\delta x^i}.
\end{displaymath}
Since the vertical Liouville vector field $\Gamma$ is orthogonal to the level hypersurfaces of the fundamental function $F$, the vector fields $\left\{\frac{\overline{\delta}}{\overline{\delta} x^a}, \xi,\frac{\overline{\partial}}{\overline{\partial} y^a}\right\}$ are tangent to these hypersurfaces in $TM^0$, so they generate the distribution $\mathcal{L}_{\Gamma}^{\perp}$. The indicatrix distribution $\mathcal{L}_{\Gamma}^{'}$ is locally generated by $\left\{\frac{\overline{\partial}}{\overline{\partial} y^a}\right\}$, and the line distribution $\mathcal{L}_{\Gamma}$ is spanned by $\Gamma$. Also, the vertical foliation has the structural bundle locally generated by $\left\{\frac{\overline{\partial}}{\overline{\partial} y^a},\Gamma\right\}$. 

Finally, we notice that if we consider the line distribution $\mathcal{L}_{\xi}$ spanned by the horizontal Liouville vector field $\xi$ then in \cite{Bej} is proved that the distribution $\mathcal{L}_{\Gamma}$, $\mathcal{L}_{\xi}$ and $\mathcal{L}_{\Gamma}\oplus\mathcal{L}_{\xi}$, respectively, are integrable and their foliations $\mathcal{F}_{\Gamma}$, $\mathcal{F}_{\xi}$ and $\mathcal{F}_{\Gamma\oplus\xi}$, respectively, are totally geodesic foliations on $(TM^0,G)$.

Also denoting by $\mathcal{L}^{'}_{\xi}$ the complement of $\mathcal{L}_{\xi}$ in $H(TM^0)$, then it is easy to see that $\mathcal{L}^{'}_{\xi}$ is locally spanned by the vector fields $\left\{\frac{\overline{\delta}}{\overline{\delta} x^a}\right\}$ and we have the decomposition
\begin{equation}
\mathcal{L}_{\Gamma}^{\perp}=\mathcal{L}_{\xi}\oplus\mathcal{L}_{\xi}^{'}\oplus\mathcal{L}_{\Gamma}^{'}.
\label{II8}
\end{equation}
\begin{remark}
In computation we shall replace the local basis $\left\{\frac{\overline{\partial}}{\overline{\partial} y^a}\right\},\,a=1,\ldots,n-1$ by the system  $\left\{\frac{\overline{\partial}}{\overline{\partial} y^i}\right\},\,i=1,\ldots,n$ taking into account relation \eqref{II6} for some easier calculations. 
\end{remark}

\section{Subfoliations in the tangent manifold of a Finsler space}
\setcounter{equation}{0}

In this section, following \cite{C-M}, we briefly recall the notion of a $(q_1,q_2)$-codimensional subfoliation on a manifold and we identify a $(n,2n-1)$-codimensional subfoliation $(\mathcal{F}_V,\mathcal{F}_{\Gamma})$ on the tangent manifold $TM^0$ of a Finsler space $(M,F)$, where $\mathcal{F}_V$ is the vertical foliation and $\mathcal{F}_{\Gamma}$ is the line foliation spanned by the vertical Liouville vector field $\Gamma$. Firstly we make a general approach about basic connections on the normal bundles related to this subfoliation and next a triple of adapted basic connections with respect to this subfoliation is given.

\begin{definition}
Let $M$ be a $n$-dimensional manifold and $TM$ its tangent bundle. A \textit{$(q_1,q_2)$-codimensional subfoliation} on $M$ is a couple $(F_1,F_2)$ of integrable subbundles $F_k$ of $TM$ of dimension $n-q_k$, $k=1,2$ and $F_2$ being at the same time a subbundle of $F_1$.
\end{definition}
For a subfoliation $(F_1,F_2)$, its normal bundle is defined as $Q(F_1,F_2)=QF_{21}\oplus QF_1$, where $QF_{21}$ is the quotient bundle $F_1/F_2$ and $QF_1$ is the usual normal bundle of $F_1$. So, an exact sequence of vector bundles
\begin{equation}
0\longrightarrow QF_{21}\stackrel{i}{\longrightarrow} QF_2\stackrel{\pi}{\longrightarrow}QF_1\longrightarrow0
\label{III1}
\end{equation}
appears in a canonical way.

Also if we consider the canonical exact sequence associated to the foliation given by an integrable subbundle $F$, namely
\begin{displaymath}
0\longrightarrow F\stackrel{i_F}{\longrightarrow}TM\stackrel{\pi_F}{\longrightarrow}QF\longrightarrow0
\end{displaymath}
then we recall that a connection $\nabla:\Gamma(TM)\times\Gamma(QF)\rightarrow\Gamma(QF)$ on the normal bundle $QF$ is said to be \textit{basic} if 
\begin{equation}
\nabla_XY=\pi_F[X,\widetilde{Y}]
\label{III2}
\end{equation} 
for any $X\in\Gamma(F)$, $\widetilde{Y}\in\Gamma(TM)$ such that $\pi_F(\widetilde{Y})=Y$.

Similarly, for a $(q_1,q_2)$-subfoliation $(F_1,F_2)$ we can consider the following exact sequence of vector bundles
\begin{equation}
0\longrightarrow F_2\stackrel{i_0}{\longrightarrow}F_1\stackrel{\pi_0}{\longrightarrow}QF_{21}\longrightarrow0
\label{III3}
\end{equation}
and according to \cite{C-M} a connection $\nabla$ on $QF_{21}$ is said to be basic with respect to the subfoliation $(F_1,F_2)$ if
\begin{equation}
\nabla_XY=\pi_0[X,\widetilde{Y}]
\label{III4}
\end{equation}
for any $X\in\Gamma(F_2)$ and $\widetilde{Y}\in\Gamma(F_1)$ such that $\pi_0(\widetilde{Y})=Y$.

\subsection{A $(n,2n-1)$-codimensional subfoliation $(\mathcal{F}_V,\mathcal{F}_{\Gamma})$ of $(TM^0,G)$}

Taking into account the discussion from the previous section, for a $n$-dimensional Finsler manifold $(M,F)$, we have on the $2n$-dimensional tangent manifold $TM^0$ a $(n,2n-1)$-codimensional foliation $(\mathcal{F}_V,\mathcal{F}_{\Gamma})$. We also notice that the metric structure $G$ on $TM^0$ given by \eqref{5} is compatible with the subfoliated structure, that is 
\begin{displaymath}
Q\mathcal{F}_V\cong H(TM^0),\,Q\mathcal{F}_{\Gamma}\cong \mathcal{L}_{\Gamma}^{\perp},\,V(TM^0)/\mathcal{L}_{\Gamma}\cong \mathcal{L}_{\Gamma}^{'}.
\end{displaymath} 
Let us consider the following exact sequences associated to the subfoliation $(\mathcal{F}_V,\mathcal{F}_{\Gamma})$ 
\begin{displaymath}
0 \longrightarrow {\mathcal{L}_{\Gamma}} \stackrel{i_0}{\longrightarrow} V(TM^0) \stackrel{\pi_0}{\longrightarrow}{\mathcal{L}_{\Gamma}^{'}} \longrightarrow 0,
\end{displaymath}
and to foliations $\mathcal{F}_V$ and $\mathcal{F}_{\Gamma}$, respectively
\begin{displaymath}
0 \longrightarrow V(TM^0) \stackrel{i_1}{\longrightarrow} T(TM^0) \stackrel{\pi_1}{\longrightarrow}H(TM^0) \longrightarrow 0,
\end{displaymath}
\begin{displaymath}
0 \longrightarrow {\mathcal{L}_{\Gamma}} \stackrel{i_2}{\longrightarrow} T(TM^0) \stackrel{\pi_2}{\longrightarrow}{\mathcal{L}_{\Gamma}^{\perp}} \longrightarrow 0,
\end{displaymath}
where $i_0,i_1,i_2$, $\pi_0,\pi_1,\pi_2$ are the canonical inclusions and projections, respectively. 

A triple $(\nabla^1, \nabla^2,\nabla)$ of basic connections on  normal bundles $\mathcal{L}_{\Gamma}^{'}$, $H(TM^0)$, $\mathcal{L}_{\Gamma}^{\perp}$, respectively, is called in \cite{C-M} \textit{adapted} to the subfoliation $(\mathcal{F}_{V},\mathcal{F}_{\Gamma})$.

Our goal is to determine such a triple of connections, adapted to this subfoliation.

By \eqref{III4} a connection $\nabla^1$ on $\mathcal{L}_{\Gamma}^{'}$ is basic with respect to the subfoliation $(\mathcal{F}_V,\mathcal{F}_{\Gamma})$ if 
\begin{equation}
\nabla^1_XZ=\pi_0[X,\widetilde{Z}],\, \forall\, X\in \Gamma(\mathcal{L}_{\Gamma}),\, \forall\, \widetilde{Z}\in \Gamma(V(TM^0)),\,\pi_0(\widetilde{Z})=Z.
\label{III5}
\end{equation}
\begin{proposition}
\label{p3.1}
A connection $\nabla^1$ on $\mathcal{L}_{\Gamma}^{'}$ is basic if and only if 
\begin{displaymath}
\nabla^1_{\Gamma}Z=[\Gamma,Z],\, \forall\, Z\in \Gamma(\mathcal{L}_{\Gamma}^{'}).
\end{displaymath}
\end{proposition}
\begin{proof}
Let $\nabla^1:V(TM^0)\times \mathcal{L}_{\Gamma}^{'}\rightarrow \mathcal{L}_{\Gamma}^{'}$ be a connection on $\mathcal{L}_{\Gamma}^{'}$ such that $\nabla^1_{\Gamma}Z=[\Gamma,Z]$. Let $X\in \Gamma(\mathcal{L}_{\Gamma})$ be a section in the structural bundle of a the line foliation $\mathcal{F}_{\Gamma}$, so its form is $X=a\Gamma$, with $a$ a differentiable function on $TM^0$. An arbitrary vertical vector field $\widetilde{Z}$ which projects into $Z\in \mathcal{L}_{\Gamma}^{'}$ is in the form
\begin{displaymath}
\widetilde{Z}=Z+b\Gamma
\end{displaymath}
with $b$ a differentiable function on $TM^0$. 

We have
\begin{eqnarray*}
[X,\widetilde{Z}]&=&[a\Gamma,Z+b\Gamma]\\
&=&a[\Gamma,Z]+(a\Gamma(b)-b\Gamma(a)-Z(a))\Gamma.
\end{eqnarray*}
According to the second relation from \eqref{II5} for any $Z=Z^i\frac{\overline{\partial}}{\overline{\partial}y^i}\in\Gamma\left(\mathcal{L}_{\Gamma}^{'}\right)$, we have
\begin{displaymath}
[\Gamma,Z]=(\Gamma(Z^i)-Z^i)\frac{\overline{\partial}}{\overline{\partial} y^i}\in\Gamma\left( \mathcal{L}_{\Gamma}^{'}\right),
\end{displaymath}
so $\pi_0[X,\widetilde{Z}]=a[\Gamma,Z]$. We also have $\nabla^1_XZ=a\nabla^1_{\Gamma}Z=a[\Gamma,Z]=\pi_0[X,\widetilde{Z}]$, hence $\nabla^1$ is a basic connection on $\mathcal{L}_{\Gamma}^{'}$.

Conversely, by the second relation from \eqref{II5}, in the adapted basis $\left\{\frac{\overline{\partial}}{\overline{\partial} y^i},\Gamma\right\}$ in $V(TM^0)$, every basic connection $\nabla^1$ on $\mathcal{L}_{\Gamma}^{'}$ is locally satisfying
\begin{equation}
\nabla^1_{\Gamma}\frac{\overline{\partial}}{\overline{\partial} y^i}=-\frac{\overline{\partial}}{\overline{\partial} y^i}
\label{III6}
\end{equation}
for any $i=1,\ldots,n$.
  
Now, if \eqref{III6} is satsfied, then
\begin{displaymath}
\nabla^1_{\Gamma}Z=\Gamma(Z^i)\frac{\overline{\partial}}{\overline{\partial} y^i}+Z^i\nabla^1_{\Gamma}\frac{\overline{\partial}}{\overline{\partial} y^i}=\Gamma(Z^i)\frac{\overline{\partial}}{\overline{\partial} y^i}-Z^i\frac{\overline{\partial}}{\overline{\partial} y^i}.
\end{displaymath}
Hence the condition \eqref{III6} is equivalent with $\nabla^1_{\Gamma}Z=[\Gamma,Z],\quad \forall Z\in \Gamma\left(\mathcal{L}_{\Gamma}^{'}\right)$.
\end{proof}
Moreover, by relation \eqref{II4}, it follows that condition \eqref{III6} has geometrical meaning. We obtain the locally characterisation:
\begin{proposition}
\label{p3.2}
A connection $\nabla^1$ on $\mathcal{L}_{\Gamma}^{'}$ is basic if and only if in an adapted local chart the relation \eqref{III6} holds. 
\end{proposition}

Now, by \eqref{III2}, a connection $\nabla^2$ on $H(TM^0)$ is  basic with respect to the vertical foliation $\mathcal{F}_V$ if 
\begin{equation}
\nabla^2_XY=\pi_1[X,\widetilde{Y}]
\label{III7}
\end{equation}
for any $X\in \Gamma(V(TM^0))$ and $\widetilde{Y}\in\Gamma(T(TM^0))$ such that $\pi_1(\widetilde{Y})=Y$
\begin{proposition}
\label{p3.3}
A connection $\nabla^2$ on $H(TM^0)$ is basic if and only if in an adapted local frame $\left\{\frac{\delta}{\delta x^i},\frac{\partial}{\partial y^i}\right\}$ on $T(TM^0)$ we have 
\begin{displaymath}
\nabla^2_{\frac{\partial}{\partial y^i}}\frac{\delta}{\delta x^i}=0
\end{displaymath}
for any $i=1,\ldots,n$.
\end{proposition}
\begin{proof} Obviously, the above condition has geometrical meaning since if $(\widetilde{U},(\widetilde{x}^{i_1},\widetilde{y}^{i_1}))$ is another local chart on $TM^0$, in $U \cap \widetilde{U} \neq \emptyset $, we have
\begin{displaymath}
\frac{\delta}{\delta \widetilde{x}^{i_1}}=\frac{\partial x^i}{\partial \widetilde{x}^{i_1}}\frac{\delta}{\delta x^i},\quad \frac{\partial}{\partial \widetilde{y}^{j_1}}=\frac{\partial x^j}{\partial \widetilde{x}^{j_1}}\frac{\partial}{\partial y^j}.
\end{displaymath}
If $\nabla^2$ is a basic connection with respect to the vertical foliation, then by definition it results 
\begin{displaymath}
\nabla^2_{\frac{\partial}{\partial y^j}}\frac{\delta}{\delta x^i}=\pi_1\left[\frac{\delta}{\delta x^i},\frac{\partial}{\partial y^j}\right]=\pi_1\left(\frac{\partial G^k_i}{\partial y^j}\frac{\partial}{\partial y^k}\right)=0.
\end{displaymath}
Conversely, let $\nabla^2:T(TM^0)\times H(TM^0)\rightarrow H(TM^0)$ be a connection on $H(TM^0)$ which locally satisfies 
\begin{displaymath}
\nabla^2_{\frac{\partial}{\partial y^i}}\frac{\delta}{\delta x^i}=0
\end{displaymath}
for any $i=1,\ldots,n$.

An arbitrary vertical vector field $X$ has local expression $X=X^i\frac{\partial}{\partial y^i}$ and a vector field $\widetilde{Y}$ whose horizontal projection is $Y=Y_h^i\frac{\delta}{\delta x^i}$ is by the form $\widetilde{Y}=Y+Y^i_v\frac{\partial}{\partial y^i}$. 

We calculate
\begin{displaymath}
\nabla^2_XY=X^i\nabla^2_{\frac{\partial}{\partial y^i}}\left(Y_h^j\frac{\delta}{\delta x^j}\right)=X^i\frac{\partial Y_h^j}{\partial y^i}\frac{\delta}{\delta x^j},
\end{displaymath}
\begin{displaymath}
\left[X,\widetilde{Y}\right]=X^i\frac{\partial Y_h^j}{\partial y^i}\frac{\delta}{\delta x^j}+\left(X^i\frac{\partial Y_v^j}{\partial y^i}-Y_h^i\frac{\delta X^j}{\delta x^i}\right)\frac{\partial}{\partial y^j}+X^iY^j_h\left[\frac{\partial}{\partial y^i},\frac{\delta}{\delta x^j}\right],
\end{displaymath}
hence the relation \eqref{III7} is verified, since $\left[\frac{\partial}{\partial y^i},\frac{\delta}{\delta x^j}\right]\in \Gamma(V(TM^0))$. So, $\nabla^2$ is a basic connection with respect to the vertical foliation $\mathcal{F}_V$.
\end{proof}

Also, by \eqref{III2}, a connection $\nabla$ on $\mathcal{L}_{\Gamma}^{\perp}$ is basic with respect to the line foliation $\mathcal{F}_{\Gamma}$ if 
\begin{equation}
\nabla_XY=\pi_2[X,\widetilde{Y}]
\label{III8}
\end{equation}
for any $X\in \Gamma(\mathcal{L}_{\Gamma})$ and $\widetilde{Y}\in \Gamma(T(TM^0))$ such that $\pi_2(\widetilde{Y})=Y$.

We have the following locally characterisation of a basic connection on $\mathcal{L}_{\Gamma}^{\perp}$:

\begin{proposition}
\label{p3.4}
A connection $\nabla$ on $\mathcal{L}_{\Gamma}^{\perp}$ is basic with respect to the line foliation $\mathcal{F}_{\Gamma}$ if and only if in an adapted local frame $\left\{\frac{\delta}{\delta x^i},\frac{\overline{\partial}}{\overline{\partial} y^i}\right\}$ on $\mathcal{L}_{\Gamma}^{\perp}$ we have 
\begin{equation}
\nabla_{\Gamma}\frac{\delta}{\delta x^i}=0,\quad \nabla_{\Gamma}\frac{\overline{\partial}}{\overline{\partial} y^i}=-\frac{\overline{\partial}}{\overline{\partial} y^i}
\label{III9}
\end{equation}
for any $i=1,\ldots,n$.
\end{proposition}
\begin{proof} Let $\nabla$ be a basic connection on $\mathcal{L}_{\Gamma}^{\perp}$. Since $\mathcal{L}_{\Gamma}^{\perp}$ is locally generated by $\left\{\frac{\delta}{\delta x^i},\frac{\overline{\partial}}{\overline{\partial} y^i}\right\}$, the condition \eqref{III8} give us the following relations:
\begin{displaymath}
\nabla_{\Gamma}\frac{\delta}{\delta x^i}=\pi_2\left[\Gamma,\frac{\delta}{\delta x^i}+f\Gamma\right]=\pi_2\left((G^j_i-\Gamma(G^j_i))\frac{\partial}{\partial y^j}+\Gamma(f)\Gamma\right)=0,
\end{displaymath}
since by homogeneity condition of Finsler structure we have \cite{Bej}, \cite{B-M}, \cite{M-A}: \begin{displaymath}
\Gamma(G^j_i)=y^k\frac{\partial G^j_i}{\partial y^k}=y^k\frac{\partial^2 G^j}{\partial y^k\partial y^i}=G^j_i,
\end{displaymath}
and
\begin{displaymath}
\nabla_{\Gamma}\frac{\overline{\partial}}{\overline{\partial} y^i}=\pi_2\left[\Gamma,\frac{\overline{\partial}}{\overline{\partial} y^i}+f\Gamma\right]=\pi_2\left(-\frac{\overline{\partial}}{\overline{\partial} y^i}+\Gamma(f)\Gamma\right)=-\frac{\overline{\partial}}{\overline{\partial} y^i}.
\end{displaymath}
Conversely, let us consider $\nabla:T(TM^0)\times \mathcal{L}_{\Gamma}^{\perp}\rightarrow \mathcal{L}_{\Gamma}^{\perp}$ be a connection on $\mathcal{L}_{\Gamma}^{\perp}$ which locally satisfies \eqref{III9}.

An arbitrary vector field $Y\in \Gamma\left(\mathcal{L}_{\Gamma}^{\perp}\right)$ is locally given by $Y=Y_h^i\frac{\delta}{\delta x^i}+Y^i\frac{\overline{\partial}}{\overline{\partial} y^i}$ and a vector field $\widetilde{Y}\in \Gamma(T(TM^0))$ which projects by $\pi_2$ in $Y$ is $\widetilde{Y}=f\Gamma+Y$. For an arbitrary vector field $X=a\Gamma\in \Gamma(\mathcal{L}_{\Gamma})$ we calculate
\begin{eqnarray*}
\nabla_{X}Y&=&\nabla_{a\Gamma}(Y_h^i\frac{\delta}{\delta x^i}+Y^i\frac{\overline{\partial}}{\overline{\partial}y^i})\\
&=&a\Gamma(Y_h^i)\frac{\delta}{\delta x^i}+aY_h^i\nabla_{\Gamma}\frac{\delta}{\delta x^i}+a\Gamma(Y^i)\frac{\overline{\partial}}{\overline{\partial} y^i}+aY^i\nabla_{\Gamma}\frac{\overline{\partial}}{\overline{\partial} y^i}\\
&=&a\Gamma(Y_h^i)\frac{\delta}{\delta x^i}+a(\Gamma(Y^i)-Y^i)\frac{\overline{\partial}}{\overline{\partial} y^i},
\end{eqnarray*}
and
\begin{eqnarray*}
\pi_2[X,\widetilde{Y}]&=&\pi_2\left(a\Gamma(f)\Gamma-f\Gamma(a)\Gamma+a\Gamma(Y_h^i)\frac{\delta}{\delta x^i}+aY^i_h\left[\Gamma,\frac{\delta}{\delta x^i}\right]\right)\\
&&+\pi_2\left(a(\Gamma(Y^i)-Y^i)\frac{\overline{\partial}}{\overline{\partial} y^i}-Y^i_h\frac{\delta a}{\delta x^i}\Gamma-Y^i\frac{\overline{\partial}a}{\overline{\partial} y^i}\Gamma\right)\\
&=&a\Gamma(Y_h^i)\frac{\delta}{\delta x^i}+a(\Gamma(Y^i)-Y^i)\frac{\overline{\partial}}{\overline{\partial} y^i},
\end{eqnarray*}
since $\left[\Gamma,\frac{\delta}{\delta x^i}\right]=0$. Thus, we have obtained that the connection $\nabla$ is a basic one.
\end{proof}

\subsection{A triple of adapted basic connections to subfoliation $(\mathcal{F}_V,\mathcal{F}_{\Gamma})$}

In \cite{B-F, Bej} there are studied several connections on $T(TM^0)$ and their relations with the vertical and fundamental Liouville distributions. The \textit{Vr\^{a}nceanu connection} $\nabla^*$ is locally given with respect to basis $\left\{\frac{\delta}{\delta x^i},\frac{\partial}{\partial y^i}\right\}$ adapted to vertical foliation, by \cite{B-F}:
\begin{displaymath}
\nabla^*_{\frac{\partial}{\partial y^j}}\frac{\partial}{\partial y^i}=C_{ij}^k\frac{\partial}{\partial y^k},\quad \nabla^*_{\frac{\delta}{\delta x^j}}\frac{\partial}{\partial y^i}=G_{ij}^k\frac{\partial}{\partial y^k},
\end{displaymath}
\begin{displaymath}
\nabla^*_{\frac{\partial}{\partial y^j}}\frac{\delta}{\delta x^i}=0,\quad \nabla^*_{\frac{\delta}{\delta x^j}}\frac{\delta}{\delta x^i}=F_{ij}^k
\frac{\delta}{\delta x^k},
\end{displaymath}
where
\begin{displaymath}
C_{ij}^k=\frac{1}{2}g^{kl}\frac{\partial g_{ij}}{\partial y^l},\quad G_{ij}^k=\frac{\partial G^k_j}{\partial y^i},
\end{displaymath}
\begin{displaymath}
F_{ij}^k=\frac{1}{2}g^{kl}\left(\frac{\delta g_{il}}{\delta x^j}+\frac{\delta g_{jl}}{\delta x^i}-\frac{\delta g_{ij}}{\delta x^l}\right).
\end{displaymath}
The restriction of Vr\^{a}nceanu connection to $T(TM^0)\times H(TM^0)$ is a connection on $H(TM^0)$ denoted by $\nabla^*_1$, which satisfies the conditions from Proposition \ref{p3.3}, so it is a basic connection on $H(TM^0)$ with respect to vertical foliation $\mathcal{F}_V$.

On the other hand in \cite{Ma} there is introduced the \textit{Vaisman connection}, also called \textit{second connection} \cite{Va, Va1}, $\nabla ^v$ on $V(TM^0)$ as it follows.  Let $v:V(TM^0)\rightarrow \mathcal{L}_{\Gamma}^{'}$, $h:V(TM^0)\rightarrow \mathcal{L}_{\Gamma}$ be the projection morphisms with respect to the decomposition $V(TM^0)=\mathcal{L}_{\Gamma}^{'}\oplus\mathcal{L}_{\Gamma}$. We search for a connection on $V(TM^0)$ with the following properties:
\begin{enumerate}
\item[a)] If $Y\in \Gamma(\mathcal{L}_{\Gamma}^{'})$ (respectively $\in \Gamma(\mathcal{L}_{\Gamma})$), then $\nabla ^v _XY\in \Gamma(\mathcal{L}_{\Gamma}^{'})$ (respectively $\in \Gamma(\mathcal{L}_{\Gamma})$), for every vertical vector field $X$.

\item[b)]  $v(T^v(X,Y))=0$, (respectively $h(T^v(X,Y))=0$) if at least one of the arguments is in $\mathcal{L}_{\Gamma}^{'}$, respectively in $\mathcal{L}_{\Gamma}$, where $T^v$ is the torsion of the connection $\nabla^v$.

\item[c)] For every $X,Y,Y'\in \Gamma(\mathcal{L}_{\Gamma}^{'})$ (respectively $\in \Gamma(\mathcal{L}_{\Gamma})$), 
\begin{displaymath}
(\nabla^ v _XG)(Y,Y')=0.
\end{displaymath}

\item[d)] Moreover, we need to give $\nabla ^v_XY$ for every $X\in H(TM^0)$ and we put the above a), b) conditions also for $X$ an horizontal vector field.
\end{enumerate}

It is proved that $\nabla ^v$ locally verifies:
\begin{displaymath}
\nabla^v_{\Gamma}\frac{\overline{\partial}}{\overline{\partial} y^i}=-\frac{\overline{\partial}}{\overline{\partial} y^i},
\end{displaymath}
which means that $\nabla^v$ is a basic connection on $\mathcal{L}_{\Gamma}^{'}$. 

Indeed, if we put 
\begin{displaymath}
\nabla^v_{\Gamma}\frac{\overline{\partial}}{\overline{\partial}y^i}=s_i^j\frac{\overline{\partial}}{\overline{\partial}y^i}\,,\,\nabla^v_{\frac{\overline{\partial}}{\overline{\partial}y^i}}\Gamma=s_i\Gamma\,,\,\nabla^v_{\Gamma}\Gamma=s\Gamma\,,\,\nabla^v_{\frac{\overline{\partial}}{\overline{\partial}y^i}}\frac{\overline{\partial}}{\overline{\partial}y^j}=s_{ij}^k\frac{\overline{\partial}}{\overline{\partial}y^k},
\end{displaymath}
\begin{displaymath}
\nabla^v_{\frac{\delta}{\delta x^i}}\frac{\overline{\partial}}{\overline{\partial}y^j}=\beta^k_{ij}\frac{\overline{\partial}}{\overline{\partial}y^k}\,,\,\nabla^v_{\frac{\delta}{\delta x^i}}\Gamma=\beta_{i}\Gamma,
\end{displaymath}
then from $T^v\left(\frac{\overline{\partial}}{\overline{\partial}y^i},\Gamma\right)=0$ we get 
$s_i=0\,,\,s_i^j=0\,\,\forall\,i\neq j\,,\,s_i^i=-1$. For details concerning the other coefficients  see \cite{Ma}.

Hence we have the basic connection $\nabla^*_1$ on $H(TM^0)$ and the basic connection $\nabla^v$ on $\mathcal{L}_{\Gamma}^{'}$. Following \cite{C-M}, we can build now a connection on $\mathcal{L}_{\Gamma}^{\perp}$ as follows:
\begin{displaymath}
\overline{\nabla}:T(TM^0)\times \mathcal{L}_{\Gamma}^{\perp}\rightarrow \mathcal{L}_{\Gamma}^{\perp}\,,\,\overline{\nabla}_XZ=\nabla^*_{1X}Z^h+\nabla^v_XZ'
\end{displaymath}
for any $Z=Z^h+Z'\in\Gamma\left( \mathcal{L}_{\Gamma}^{\perp}\right)=\Gamma\left(H(TM^0)\right)\oplus\Gamma\left(\mathcal{L}_{\Gamma}^{'}\right)$ and $X\in \Gamma(T(TM^0))$.

By direct calculus we have
\begin{displaymath}
\overline{\nabla}_{\Gamma}\frac{\delta}{\delta x^i}=\nabla^*_{1\Gamma}\frac{\delta}{\delta x^i}=0, \quad \overline{\nabla}_{\Gamma}\frac{\overline{\partial}}{\overline{\partial} y^i}=\nabla^v_{\Gamma}\frac{\overline{\partial}}{\overline{\partial} y^i}=-\frac{\overline{\partial}}{\overline{\partial} y^i},
\end{displaymath}
so the connection $\overline{\nabla}$ satisfies conditions from Proposition \ref{p3.4}, hence it is a basic connection on $\mathcal{L}_{\Gamma}^{\perp}$. 

Now, for the $(n,2n-1)$-codimensional subfoliation $(\mathcal{F}_V,\mathcal{F}_{\Gamma})$ we can consider the exact sequence \eqref{III1}, namely
\begin{displaymath}
0\longrightarrow\mathcal{L}_{\Gamma}^{'}\stackrel{i}{\longrightarrow}\mathcal{L}_{\Gamma}^{\perp}\stackrel{\pi}{\longrightarrow}H(TM^0)\longrightarrow0
\end{displaymath}
and the general theory of $(q_1,q_2)$-codimensional foliations, see \cite{C-M}, leads to the following results:
\begin{proposition}
\label{p3.5}
For the triple basic connections $(\nabla^v,\nabla^*_1,\overline{\nabla})$ on $\mathcal{L}_{\Gamma}^{'}$, $H(TM^0)$ and $\mathcal{L}_{\Gamma}^{\perp}$, respectively, we have
\begin{displaymath}
i\left(\nabla^v_XY\right)=\overline{\nabla}_Xi(Y)\,,\,\pi\left(\overline{\nabla}_XZ\right)=\nabla^*_{1X}\pi(Z)\,,\,\,\forall\,Y\in\Gamma\left(\mathcal{L}_{\Gamma}^{'}\right),\,Z\in\Gamma\left(\mathcal{L}_{\Gamma}^{\perp}\right).
\end{displaymath}
\end{proposition}
\begin{proposition}
\label{p3.6}
The triple of basic connections $(\nabla^v,\nabla^*_1,\overline{\nabla})$ is adapted to the subfoliation $(\mathcal{F}_V,\mathcal{F}_{\Gamma})$ of $(TM^0,G)$.
\end{proposition}
\begin{remark}
We notice that the connection $\nabla=\nabla^v\oplus\nabla^*_1$ is a basic connection on $Q(\mathcal{F}_V,\mathcal{F}_{\Gamma})=\mathcal{L}_{\Gamma}^{'}\oplus H(TM^0)=\mathcal{L}_{\Gamma}^{\perp}$ and its curvature $K$ satisfy
\begin{equation}
K(X,Y)=0\,,\,\,\forall\,X,Y\in\Gamma(\mathcal{L}_{\Gamma}).
\label{III10}
\end{equation}
The Bott's obstruction theorem for subfoliations (Theorem 3.9 \cite{C-M}) leads to:
\begin{corollary}
Let $\mathcal{P}_1$ and $\mathcal{P}_2$ be homogeneous polynomials in the real Pontryagin classes of $\mathcal{L}_{\Gamma}^{'}$ and $H(TM^0)$ respectively, of degree $l_k$, $k=1,2$. If at least one of the inequalities $l_2>2n$, $l_1+l_2>4n-2$ is satisfied, then $\mathcal{P}_1\cdot\mathcal{P}_2=0$.
\end{corollary}
\end{remark}

\noindent
Adelina Manea and Cristian Ida\\
Department of Mathematics and Computer Science\\
University Transilvania of Bra\c{s}ov\\
Address: Bra\c{s}ov 500091, Str. Iuliu Maniu 50, Rom\^{a}nia\\
email: \textit{amanea28@yahoo.com; cristian.ida@unitbv.ro}\\
\\

\smallskip

\end{document}